\documentclass[12pt]{article}

\usepackage[T1]{fontenc}
\usepackage{amsmath}
\usepackage{amssymb}
\usepackage{amsfonts}
\usepackage{fullpage}
\usepackage{latexsym,amssymb,amsmath}
\usepackage{euscript}

\usepackage{tikz}
\usetikzlibrary{patterns}

\usepackage[colors]{optsys}

\allowdisplaybreaks

\newcommand{\NATURE}{\Omega}
\renewcommand{\nature}{\omega}
\newcommand{\BELIEF}{\Delta\np{\NATURE}}
\newcommand{\belief}{p}

\newcommand{\SecondOrderBelief}{Q}

\newcommand{\MaximalExpectedUtility}{maximal expected utility}

\renewcommand{\fonctiondeux}{h}
\newcommand{\NEUTRAL}{A}
\newcommand{\neutral}{a}

\newcommand{\NEUTRALpayoff}{U}

\newcommand{\MORE}{{M}}

\newcommand{\LESS}{{L}}

\newcommand{\less}{l}
\newcommand{\MEDIUM}{{N}}

\newcommand{\medium}{n}
\renewcommand{\ValueFunction}[1]{V_{#1}}
\newcommand{\VoI}{\mathrm{V\!o\!I}}

\renewcommand{\Primal}{X}

\newcommand{\signed}{y}
\renewcommand{\sequence}[2]{\np{#1}_{#2}}           
\newcommand{\scalpro}[2]{\left\langle#2 \mid \:#1\right\rangle}  

\title{Adding Decision Problem\\ Makes Information More Valuable}

\author{Michel \textsc{De Lara}\thanks{%
CERMICS, ENPC, Institut Polytechnique de Paris, CNRS, Marne-la-Vall\'ee, France
E-mail: \texttt{michel.delara@enpc.fr}}}

\begin{document}

\maketitle

    \begin{abstract}%
We consider decision-making under incomplete information about an unknown state
of nature. We show that a decision problem yields a higher value of information
than another, uniformly across information structures, if and only if it is obtained by adding
an independent, parallel decision problem.
    \end{abstract}%

\textbf{Keywords:} value of information, information more valuable, separable
additive utility.





 \section{Introduction}
 \label{Introduction}

 Information has value. But what makes information more valuable?
 We consider decision-making under incomplete information about an unknown state
 of nature.
A decision problem is given by a utility function defined over the Cartesian
product of an action set with a set of states of nature.
 Information is traditionally modeled as a distribution of posterior beliefs on states of nature
(see \cite{Blackwell:1953} and references therein), also called information
structure (or system \cite{Marschak-Miyasawa:1968}).
The value of information (VoI) is the maximal expected utility with distribution of
posterior beliefs minus the one with only prior belief (barycenter). 

In this paper, we focus on the little studied subject of interpersonal comparison of VoIs
--- a question at the center of our 2009 working paper~\cite{DeLara:2009},
and recently studied in \cite{whitmeyer2024makinginformationvaluable}:
what makes information more valuable, for any information structure,
in one decision problem than in another?
The heart of \cite{whitmeyer2024makinginformationvaluable} is finding economic transformations (sufficient conditions) which increase
the value of information, but also what transformations increase the demand for
information.
 To the difference with \cite{whitmeyer2024makinginformationvaluable}, we
 provide an ``if and only if'' answer as follows: 
 a decision problem yields a higher value of information than another, uniformly
 across information structures,
 if and only the former is obtained from the latter by adding
 an independent, parallel decision problem.
Otherwise said, a decision problem that provides more valuable information can
be interpreted as asking the decision maker to solve the original decision problem
(by choosing an action) and, separately, to solve another decision problem (by
choosing another parallel, independent action) with her payoff being the sum
across the two.

We sketch the proof as follows.
The (convex) value function of a decision problem is the maximal expected utility as a
function of prior belief. 
 \cite[Theorem~3.1]{whitmeyer2024makinginformationvaluable}
 establishes\footnote{%
A result that can also be found in \cite[Proposition~0]{Jones-Ostroy:1984},
and is possibly folk knowledge in the economic literature.
\label{it:folk}
}
   that more valuable information
 is equivalent to convexity of the difference~$f-g$ of the two (convex) value
 functions~$f,g$ associated with two decision problems.
 Now, modulo technical issues (e.g.,
 \cite{Machina:1984,Corrao-Fudenberg-Levine:2024} and see Appendix~\ref{Proof_of_Theorem}) convex functions over beliefs are equivalent
to decision problems. Thus, the resulting convex function~$h=f-g$ (over
 beliefs) is itself a value function for some decision problem --- 
$h$ is a supremum of linear functions, and each of them can
be identified with an action (that is, a vector of state-dependent payoffs) ---
which is exactly the decision problem we need to add, with additively-separable payoffs, to the
original (giving the value function $f=g+h$).
\bigskip

The paper provides a formal statement and a mathematical proof as follows.
In Sect.~\ref{information_more_valuable}, we define decision problem, value
function, information, value of information. Then, we state our main result: 
more valuable information, regardless of what the information
is, is equivalent to a decision maker being afforded access to some additional decision problem (on top of the original
one) and her utility being additively separable across the two problems.
The proof is given in Appendix~\ref{Proof_of_Theorem}.


\section{More valuable information}
\label{information_more_valuable}

We denote $\RR$ the set of real numbers, $\RR_{+} =
\ClosedIntervalOpen{0}{+\infty} $ and
$\barRR = \RR \cup \na{-\infty,+\infty} $.
We consider a nonempty finite set~$\NATURE$, that represents states of nature.
We identify the elements of the simplex
\begin{equation}
\BELIEF = \defset{\belief \in \RR_{+}^{\NATURE}}{\sum_{\nature \in \NATURE} \belief\np{\nature}=1} 
\subset \RR^{\NATURE}   
\end{equation}
with {probability distributions} over~$\NATURE$, also called \emph{beliefs}.

A \emph{decision problem} (on~$\NATURE$) is given by 
a \emph{utility function}~$\NEUTRALpayoff \colon \NEUTRAL\times
\NATURE\to \RR$, where $\NEUTRAL$ is a nonempty \emph{set of decisions/actions}.
Under belief~$\belief \in \BELIEF$, the decision maker chooses 
an action~$\neutral\in\NEUTRAL$ that maximizes~$\sum_{\nature \in \NATURE}\belief\np{\nature} \NEUTRALpayoff(\neutral,\nature)$.
The resulting \MaximalExpectedUtility\ gives rise to the so-called (convex) \emph{value function}
\begin{equation}
  \ValueFunction{\NEUTRALpayoff} \colon \BELIEF \to \barRR \eqsepv
\ValueFunction{\NEUTRALpayoff}(\belief)
= \sup_{\neutral\in\NEUTRAL}
\sum_{\nature\in\NATURE} \belief\np{\nature} \NEUTRALpayoff(\neutral,\nature)
\eqsepv \forall \belief \in \BELIEF 
  \eqfinp
    \label{eq:value_function_utility_classic}
\end{equation}
In all that follows, we will only consider decision problems for which 
the value function~\( \ValueFunction{\NEUTRALpayoff} \) in~\eqref{eq:value_function_utility_classic}
takes finite values (on~\( \BELIEF \)).
In that case, as \( \BELIEF \) is bounded polyhedral, we get that
\(\ValueFunction{\NEUTRALpayoff}\colon \BELIEF \to \RR \) is continuous and
bounded, by~\cite{Gale-Klee-Rockafellar:1968}. 


Assuming that decision-makers react to information in a Bayesian way,
\emph{information} may be modeled as a \emph{distribution \( \SecondOrderBelief \in
\Delta\bp{\BELIEF} \)  of posterior beliefs}, with the barycenter
\( \int_{\BELIEF}\belief \dd\SecondOrderBelief(\belief)
\in {\BELIEF} \) being the \emph{prior belief}. 
%

Consider a decision problem~$\NEUTRALpayoff\colon\NEUTRAL\times\NATURE\to\RR$ with finite value function~\( \ValueFunction{\NEUTRALpayoff} \).
The \emph{value of information} --- that is, the value to a decision-maker
facing decision problem~$\NEUTRALpayoff\colon\NEUTRAL\times\NATURE\to\RR$
of having information~\( \SecondOrderBelief \in \Delta\bp{\BELIEF}\) --- is
\begin{equation}
  \VoI_{\NEUTRALpayoff}\np{\SecondOrderBelief}=
\int_{\BELIEF}\ValueFunction{\NEUTRALpayoff}\np{\belief} \dd\SecondOrderBelief\np{\belief}
-\ValueFunction{\NEUTRALpayoff}\Bp{\int_{\BELIEF}\belief
  \dd\SecondOrderBelief\np{\belief}} 
\eqsepv \forall \SecondOrderBelief \in \Delta\bp{\BELIEF}
 \eqfinp
\label{eq:VoI_classic}
\end{equation}
%
The quantity in~\eqref{eq:VoI_classic} is well defined and satisfies
\begin{equation}
  +\infty > \VoI_{\NEUTRALpayoff}\np{\SecondOrderBelief} \geq 0
 \eqsepv \forall \SecondOrderBelief \in \Delta\bp{\BELIEF}
 \eqfinp
\label{eq:VoI_classic_properties}
\end{equation}
Indeed, on the one hand, as the value function~\( \ValueFunction{\NEUTRALpayoff} \) in~\eqref{eq:value_function_utility_classic}
takes finite values, it is continuous and bounded by~\cite{Gale-Klee-Rockafellar:1968}, hence
\( \int_{\BELIEF}\ValueFunction{\NEUTRALpayoff}\np{\belief}
\dd\SecondOrderBelief\np{\belief} \) is well defined and finite.
On the other hand, \( \VoI_{\NEUTRALpayoff}\np{\SecondOrderBelief} \geq 0 \)
since the value function~\( \ValueFunction{\NEUTRALpayoff} \) is convex (as a
supremum of linear functions).




\begin{theorem}
\label{th:more_valuable_information_equivalent_additively_separable}
Consider two decision problems~$\MORE$ and~$\LESS$, given by the two 
utility functions
$\NEUTRALpayoff_{\MORE} \colon \NEUTRAL_{\MORE}\times\NATURE\to \RR$
and
$\NEUTRALpayoff_{\LESS} \colon \NEUTRAL_{\LESS}\times\NATURE\to \RR$.
Suppose that the value functions \( \ValueFunction{\NEUTRALpayoff_{\MORE}} \)
and \( \ValueFunction{\NEUTRALpayoff_{\LESS}} \)
in~\eqref{eq:value_function_utility_classic}
take finite values (on~\( \BELIEF \)).
The following statements are equivalent.
\begin{enumerate}
\item
\label{it:values_information_more}
  Decision problem~$\MORE$ values information more than decision
  problem~$\LESS$, in the sense that 
\begin{equation}
  \VoI_{\NEUTRALpayoff_{\MORE}}\np{\SecondOrderBelief}
  \geq
  \VoI_{\NEUTRALpayoff_{\LESS}}\np{\SecondOrderBelief}
\eqsepv \forall \SecondOrderBelief \in \Delta\bp{\BELIEF}
  \eqfinp 
\end{equation}
\item
\label{it:additively_separable}
  There exists a decision problem $\NEUTRALpayoff_{\MEDIUM} \colon
  \NEUTRAL_{\MEDIUM}\times \NATURE\to \RR$ --- 
whose value function \( \ValueFunction{\NEUTRALpayoff_{\MEDIUM}} \)
in~\eqref{eq:value_function_utility_classic}
take finite values (on~\( \BELIEF \)) ---   
such that, if we form the new decision problem~$\widetilde{\MORE}$ with 
Cartesian product decision set~$\NEUTRAL_{\widetilde{\MORE}}=\NEUTRAL_{\LESS}\times\NEUTRAL_{\MEDIUM}$ and with
additively separable 
utility function~$\NEUTRALpayoff_{\widetilde{\MORE}}=\NEUTRALpayoff_{\LESS}+\NEUTRALpayoff_{\MEDIUM} 
\colon \np{\NEUTRAL_{\LESS}\times\NEUTRAL_{\MEDIUM}}\times \NATURE\to \RR$
defined by
\begin{equation}
 \NEUTRALpayoff_{\widetilde{\MORE}}\bp{\np{\less,\medium},\nature}
  =\NEUTRALpayoff_{\LESS}(\less,\nature)+\NEUTRALpayoff_{\MEDIUM}(\medium,\nature) \eqsepv
\forall \np{\less,\medium}\in\NEUTRAL_{\LESS}\times\NEUTRAL_{\MEDIUM}
  \eqsepv \forall \nature \in \NATURE 
  \eqfinv
\label{eq:additively_separable}
\end{equation}
then \( \widetilde{\MORE} \) and \( \MORE \) are essentially the same
decision problems, in the sense that their value functions~\eqref{eq:value_function_utility_classic} are identical:
\begin{equation}
  \ValueFunction{\NEUTRALpayoff_{\widetilde{\MORE}}} =
  \ValueFunction{\NEUTRALpayoff_{\MORE}}
  \eqfinp
\label{eq:value_functions_identical}
\end{equation}
\end{enumerate}
\end{theorem}

\begin{proof}
The proof is given in Appendix~\ref{Proof_of_Theorem}.
\end{proof}

\section{Conclusion}
\label{Conclusion}

This question ``for a given economic agent, what changes in decision variables and utility function
 make information more valuable?'' is studied in
 \cite{whitmeyer2024makinginformationvaluable}.
 The results within indicate  that, by extending the set of options with new decision variables,
 there is no systematic comparison regarding the value of information;
 stringent conditions are required to make information more valuable 
 by adding (in the sense of union) decision variables.

 Our Theorem~\ref{th:more_valuable_information_equivalent_additively_separable}
 furnishes an explanation for the above observation.
 Systematic comparison regarding the value of information is obtained if and
 only if the set of options is extended by \emph{multiplying} decision
 variables and \emph{adding utility}.
 Aside the original set of options, the economic agent has to select
 \emph{also} in a new set of options, leading to a pair of decisions (instead of
 a single one in the original problem), and then adding utilities.

 \bigskip

\textbf{Acknowledgments.} I thank Carlos Al\'os-Ferrer for his economic
comments.
I especially thank the Editor of \emph{Journal of Political Economy}, as well as
four Reviewers for their insightful comments and recommendations.

\appendix

\section{Proof of Theorem~\ref{th:more_valuable_information_equivalent_additively_separable}}
\label{Proof_of_Theorem}

We denote by $\RR$ the set of real numbers,
$\barRR = \ClosedIntervalClosed{-\infty}{+\infty} = \RR \cup \na{-\infty,+\infty} $,
$\RR_{+} = \ClosedIntervalOpen{0}{+\infty} $,
$\RR_{++}=\OpenIntervalOpen{0}{+\infty}$.
We consider a nonempty finite set~$\NATURE$.
Let \( \Primal \subset \RR^{\NATURE} \) be a nonempty set.
For any function \( \fonctionprimal \colon \Primal \to \barRR \),
its \emph{epigraph} is
\( \epigraph\fonctionprimal= \defset{
  \np{\primal,t}\in\Primal\times\RR}%
{\fonctionprimal\np{\primal} \leq t} \subset \RR^{\NATURE}\times\RR \),
its \emph{effective domain} is
\( \dom\fonctionprimal= \defset{\primal\in\Primal}{%
  \fonctionprimal\np{\primal} <+\infty} \).
A function
\( \fonctionprimal \colon \Primal \to \barRR \) is said to be \emph{convex}
if its epigraph is a convex subset of \( \RR^{\NATURE}\times\RR \), 
\emph{proper} if it never takes the
value~$-\infty$ and that \( \dom\fonctionprimal \not = \emptyset \),
\emph{lower semi continuous (\lsc)} if its epigraph is
a closed subset of \( \RR^{\NATURE}\times\RR \),
\emph{positively homogeneous} if its epigraph is a cone.
%
%
The scalar product between $\primal, \signed \in \RR^{\NATURE}$ is 
\(  \scalpro{\signed}{\primal} = 
\sum_{\nature \in \NATURE} \primal\np{\nature} \signed\np{\nature} \). 
The \emph{support function} of a subset
$\Primal\subset\RR^{\NATURE}$ is the closed convex
positively homogeneous function
\cite[Chapter~V, Definition~2.1.1]{Hiriart-Urruty-Lemarechal-I:1993} 
defined by
\begin{equation}
\SupportFunction{\Primal} \colon \RR^{\NATURE} \to \barRR \eqsepv    
    \SupportFunction{\Primal}\np{\signed}
    =
      \sup_{\primal\in \Primal} \scalpro{\signed}{\primal} \eqsepv \forall \signed\in\RR^{\NATURE}
      \eqfinp
      \label{eq:support_function}
\end{equation}

\begin{proof}
  We start by observing --- by comparing~\eqref{eq:value_function_utility_classic}
  and~\eqref{eq:support_function} --- that the value function~\( \ValueFunction{\NEUTRALpayoff} \)
associated with a decision problem~$\NEUTRALpayoff\colon\NEUTRAL\times\NATURE\to\RR$,
can be written as
\begin{equation}
  \ValueFunction{\NEUTRALpayoff}=\SupportFunction{\nc{\NEUTRALpayoff}} \mtext{
    where } \nc{\NEUTRALpayoff}=\Ba{ \NEUTRALpayoff(\neutral,\cdot),
  \neutral\in\NEUTRAL} 
  =\Ba{ \bp{\NEUTRALpayoff(\neutral,\nature)}_{\nature\in\NATURE}, \neutral\in\NEUTRAL} \subset \RR^{\NATURE}
  \eqfinp
    \label{eq:value_function_support_function}
\end{equation}

\noindent$\bullet$
Item~\ref{it:additively_separable} $\implies$
Item~\ref{it:values_information_more}. 

For any {second-order belief} \( \SecondOrderBelief \in\Delta\bp{\BELIEF}\), we get that 
\begin{subequations}
\begin{align*}
  \VoI_{{\MORE}}\np{\SecondOrderBelief} 
  &=
\int_{\BELIEF}\ValueFunction{\NEUTRALpayoff_{{\MORE}}}\np{\belief} \dd\SecondOrderBelief\np{\belief}
-\ValueFunction{\NEUTRALpayoff_{\MORE}}\Bp{\int_{\BELIEF}\belief
    \dd\SecondOrderBelief\np{\belief}}
    \intertext{by definition~\eqref{eq:VoI_classic} of the value of information~\(\VoI_{{\MORE}}\) }
  &=
\int_{\BELIEF}\ValueFunction{\NEUTRALpayoff_{\widetilde{\MORE}}}\np{\belief} \dd\SecondOrderBelief\np{\belief}
-\ValueFunction{\NEUTRALpayoff_{\widetilde{\MORE}}}\Bp{\int_{\BELIEF}\belief
    \dd\SecondOrderBelief\np{\belief}}
    \tag{as \(  \ValueFunction{\NEUTRALpayoff_{\widetilde{\MORE}}} =
  \ValueFunction{\NEUTRALpayoff_{\MORE}} \) by assumption~\eqref{eq:value_functions_identical}}
  \\
  &=
    \int_{\BELIEF}\SupportFunction{\nc{\NEUTRALpayoff_{\widetilde{\MORE}}}}\np{\belief}
    \dd\SecondOrderBelief\np{\belief}
    -\SupportFunction{\nc{\NEUTRALpayoff_{\widetilde{\MORE}}}}
    \Bp{\int_{\BELIEF}\belief\dd\SecondOrderBelief\np{\belief}}
    \tag{as \(\ValueFunction{\NEUTRALpayoff_{\MORE}}=\SupportFunction{\nc{\NEUTRALpayoff_{\widetilde{\MORE}}}}\)
    by~\eqref{eq:value_function_support_function}}
  \\
  &=
  \int_{\BELIEF}\SupportFunction{\nc{\NEUTRALpayoff_{\LESS}+\NEUTRALpayoff_{\MEDIUM}}}\np{\belief}
   \dd\SecondOrderBelief\np{\belief}
  -\SupportFunction{\nc{\NEUTRALpayoff_{\LESS}+\NEUTRALpayoff_{\MEDIUM}}}
      \Bp{\int_{\BELIEF}\belief\dd\SecondOrderBelief\np{\belief}}
  \tag{as $\NEUTRALpayoff_{\widetilde{\MORE}}=\NEUTRALpayoff_{\LESS}+\NEUTRALpayoff_{\MEDIUM}$
    by~\eqref{eq:additively_separable}}
 \\
  &=
  \int_{\BELIEF}\SupportFunction{\nc{\NEUTRALpayoff_{\LESS}}+\nc{\NEUTRALpayoff_{\MEDIUM}}}\np{\belief}
   \dd\SecondOrderBelief\np{\belief}
  -\SupportFunction{\nc{\NEUTRALpayoff_{\LESS}}+\nc{\NEUTRALpayoff_{\MEDIUM}}}
      \Bp{\int_{\BELIEF}\belief\dd\SecondOrderBelief\np{\belief}}
    \intertext{as \( \nc{\NEUTRALpayoff_{\LESS}+\NEUTRALpayoff_{\MEDIUM}}=
\nc{\NEUTRALpayoff_{\LESS}}+\nc{\NEUTRALpayoff_{\MEDIUM}} \) follows
    easily from~\eqref{eq:value_function_support_function}}
  &=
  \int_{\BELIEF}\bp{\SupportFunction{\nc{\NEUTRALpayoff_{\LESS}}}+\SupportFunction{\nc{\NEUTRALpayoff_{\MEDIUM}}}}\np{\belief}
   \dd\SecondOrderBelief\np{\belief}
  -\bp{\SupportFunction{\nc{\NEUTRALpayoff_{\LESS}}}+\SupportFunction{\nc{\NEUTRALpayoff_{\MEDIUM}}}}
      \Bp{\int_{\BELIEF}\belief\dd\SecondOrderBelief\np{\belief}}
    \intertext{as \(
    \SupportFunction{\nc{\NEUTRALpayoff_{\LESS}}+\nc{\NEUTRALpayoff_{\MEDIUM}}}=
    \SupportFunction{\nc{\NEUTRALpayoff_{\LESS}}}+\SupportFunction{\nc{\NEUTRALpayoff_{\MEDIUM}}}\)
    \cite[Chapter~V, Section~2]{Hiriart-Urruty-Lemarechal-I:1993}}
   &=
 \underbrace{ \int_{\BELIEF}\SupportFunction{\nc{\NEUTRALpayoff_{\LESS}}}\np{\belief}\dd\SecondOrderBelief\np{\belief}
-\SupportFunction{\nc{\NEUTRALpayoff_{\LESS}}}
      \Bp{\int_{\BELIEF}\belief\dd\SecondOrderBelief\np{\belief}}}_{\VoI_{\LESS}\np{\SecondOrderBelief}}
     +
     \underbrace{ \int_{\BELIEF}\SupportFunction{\nc{\NEUTRALpayoff_{\MEDIUM}}}\np{\belief}
     \dd\SecondOrderBelief\np{\belief}
       -\SupportFunction{\nc{\NEUTRALpayoff_{\MEDIUM}}}
      \Bp{\int_{\BELIEF}\belief\dd\SecondOrderBelief\np{\belief}}}_{\VoI_{\MEDIUM}\np{\SecondOrderBelief}\geq 0 \textrm{ by~\eqref{eq:VoI_classic_properties}}}
  \intertext{by~\eqref{eq:value_function_support_function} and
  definition~\eqref{eq:VoI_classic} of the value of information, and as all
  quantities are finite}
  &\geq
    \VoI_{\LESS}\np{\SecondOrderBelief}
    \eqfinp 
\end{align*}
\end{subequations}
 \bigskip 

\noindent$\bullet$
Item~\ref{it:values_information_more} $\implies$ Item~\ref{it:additively_separable}.

It is well known that \(\VoI_{\NEUTRALpayoff_{\MORE}}\geq\VoI_{\NEUTRALpayoff_{\LESS}}\) if and only if
\(\VoI_{\NEUTRALpayoff_{\MORE}}-\VoI_{\NEUTRALpayoff_{\LESS}}\) is a convex
function on~\(\Delta\bp{\BELIEF}\) (\cite[Proposition~0]{Jones-Ostroy:1984},
\cite[Footnote~9]{Chambers-Liu-Rehbeck:2020},
\cite[Theorem~3.1]{whitmeyer2024makinginformationvaluable}, see also
Footnote~\ref{it:folk}).
Indeed, we have that (all the integrals below are
finite by assumption  that the value functions \(\ValueFunction{\NEUTRALpayoff_{\MORE}}\)
and \( \ValueFunction{\NEUTRALpayoff_{\LESS}} \) take finite values (on~\( \BELIEF \)))
\begin{align*}
  &
    \VoI_{\NEUTRALpayoff_{\MORE}}\np{\SecondOrderBelief}
    \geq\VoI_{\NEUTRALpayoff_{\LESS}}\np{\SecondOrderBelief}
\eqsepv \forall \SecondOrderBelief \in \Delta\bp{\BELIEF}    
    \\
  \iff&
 \int_{\BELIEF}\ValueFunction{\NEUTRALpayoff_{\MORE}}\np{\belief} \dd\SecondOrderBelief\np{\belief}
-\ValueFunction{\NEUTRALpayoff_{\MORE}}\bp{\int_{\BELIEF}\belief \dd\SecondOrderBelief\np{\belief}}
  \\
  & \geq
    \int_{\BELIEF}\ValueFunction{\NEUTRALpayoff_{\LESS}}\np{\belief} \dd\SecondOrderBelief\np{\belief}
-\ValueFunction{\NEUTRALpayoff_{\LESS}}\bp{\int_{\BELIEF}\belief
  \dd\SecondOrderBelief\np{\belief}}
\eqsepv \forall \SecondOrderBelief \in \Delta\bp{\BELIEF}    
    \\
  \iff&
        \int_{\BELIEF}\bp{\ValueFunction{\NEUTRALpayoff_{\MORE}}-\ValueFunction{\NEUTRALpayoff_{\LESS}}}
        \np{\belief} \dd\SecondOrderBelief\np{\belief}
        \geq \bp{\ValueFunction{\NEUTRALpayoff_{\MORE}}-\ValueFunction{\NEUTRALpayoff_{\LESS}}}
        \bp{\int_{\BELIEF}\belief \dd\SecondOrderBelief\np{\belief}}
        \eqsepv \forall \SecondOrderBelief \in \Delta\bp{\BELIEF}
    \\
  \iff&
 \ValueFunction{\NEUTRALpayoff_{\MORE}}-\ValueFunction{\NEUTRALpayoff_{\LESS}}
        \text{ is a convex function on~}\BELIEF
        \eqfinp 
\end{align*}
Now, we set 
\(
\fonctiondeux=\ValueFunction{\NEUTRALpayoff_{\MORE}}-\ValueFunction{\NEUTRALpayoff_{\LESS}}\),
which is convex, and takes finite values as the value functions \(
\ValueFunction{\NEUTRALpayoff_{\MORE}} \)
and \( \ValueFunction{\NEUTRALpayoff_{\LESS}} \)
in~\eqref{eq:value_function_utility_classic}
take finite values by assumption. 
We extend $\fonctiondeux\colon \BELIEF \to \RR $ into 
\( \widetilde{\fonctiondeux} \colon \RR^{\NATURE}\to \barRR \) by
\begin{equation}
  \widetilde{\fonctiondeux}(0)=0 \eqsepv 
  \widetilde{\fonctiondeux}(\signed) =  \scalpro{\signed}{1}
  \fonctiondeux\Bp{\frac{\signed}{\scalpro{\signed}{1}}}
  \eqsepv \forall \signed \in \RR_{+}^{\NATURE}\setminus\na{0}
  \eqsepv \widetilde{\fonctiondeux}(\signed) = +\infty  \eqsepv \forall \signed \in
\RR^{\NATURE}\setminus\RR_{+}^{\NATURE}
        \eqfinv
        \label{eq:positively_homogeneous_extension_formula}
      \end{equation}
      where \( \scalpro{\signed}{1} = \sum_{\nature \in \NATURE}
      \signed\np{\nature}\)
(which is $>0$ when \(\signed \in \RR_{+}^{\NATURE}\setminus\na{0}\)).
 The function~\( \widetilde{\fonctiondeux} \) is the \emph{positively homogeneous extension} of the function
 \( \fonctiondeux \): indeed, it is positively
 homogeneous since it is straightforward to observe that its epigraph is
\begin{equation}
  \epigraph\widetilde{\fonctiondeux} = \np{\RR_{++} \epigraph\fonctiondeux} \cup
  \np{\na{0}\times\RR_{+}}
  \eqfinv
\label{eq:epigraph_fonctiondeux}
\end{equation}
which is a cone.
We are going to show that the function \( \widetilde{\fonctiondeux}\) is proper convex \lsc.
As the function \( \fonctiondeux \colon \BELIEF \to \RR \) is proper (it
takes finite values),
it is obvious by~\eqref{eq:positively_homogeneous_extension_formula}
that the function \( \widetilde{\fonctiondeux} \colon \RR^{\NATURE}\to \barRR \) is proper.
We are going to prove that the conic epigraph~\( \epigraph\widetilde{\fonctiondeux} \)
given by~\eqref{eq:epigraph_fonctiondeux} is closed convex.
For that purpose, we show that
\( \epigraph\widetilde{\fonctiondeux} = \overline{\RR_{++} \epigraph\fonctiondeux} \)
by two inclusions.

To prove the inclusion \( \epigraph\widetilde{\fonctiondeux} \subset \overline{\RR_{++} \epigraph\fonctiondeux} \),
it suffices to show that
\( \na{0}\times \RR_{+} \subset \overline{\RR_{++} \epigraph\fonctiondeux} \),
because of~\eqref{eq:epigraph_fonctiondeux}.
As the function \( \fonctiondeux \colon \BELIEF \to \RR \) takes finite values,
any \( \belief\in\BELIEF \) is such that 
\( -\infty < \fonctiondeux(\belief) < +\infty \).
Then, for any \( t\in \RR_{+} \), we get that 
  \( \RR_{++}\epigraph\fonctiondeux \ni
  \bp{ \frac{1}{n} \belief, \frac{1}{n} \fonctiondeux\np{\belief}+t} 
  \to_{n\to +\infty} \np{0,t} \in \na{0}\times \RR_{+} \). 
  Thus, we have shown that
  \( \na{0}\times \RR_{+} \subset \overline{\RR_{++} \epigraph\fonctiondeux} \),
  hence that 
  \( \epigraph\widetilde{\fonctiondeux} \subset \overline{\RR_{++} \epigraph\fonctiondeux} \)
  by~\eqref{eq:epigraph_fonctiondeux}.

  To prove the reverse inclusion
  \( \epigraph\widetilde{\fonctiondeux} \supset  \overline{\RR_{++} \epigraph\fonctiondeux} \), 
  we consider sequences 
\( \sequence{\belief_n}{n\in\NN} \subset \BELIEF \), 
\( \sequence{\lambda_n}{n\in\NN} \subset \RR_{++} \), 
and \( \sequence{t_n}{n\in\NN} \subset \RR_{+} \)
such that the sequence
\( \Bp{\bp{\lambda_n\belief_n,\lambda_n\fonctiondeux(\belief_n)+t_n}}_{n\in\NN}
\subset \RR_{++} \epigraph\fonctiondeux \) 
converges to \( \np{\bar\signed,\bar z} \in \RR^{\NATURE}\times \RR \).
Notice that \( \np{\bar\signed,\bar z} \in \RR_{+}^{\NATURE}\times \RR \),
as $\RR_{++}\BELIEF = \RR_{+}^{\NATURE}\setminus\na{0}$, hence
$\overline{\RR_{++}\BELIEF} = \RR_{+}^{\NATURE}$.
We are going to show that \( \np{\bar\signed,\bar z} \in \epigraph\widetilde{\fonctiondeux} \)
given by~\eqref{eq:epigraph_fonctiondeux}.
From \( \lambda_n\belief_n \to_{n\to +\infty} \bar\signed \) and
\( \scalpro{\belief_n}{1}= \sum_{\nature \in \NATURE}\belief_n\np{\nature}=1 \), we get that \( \lambda_n = \scalpro{\lambda_n\belief_n}{1} \to_{n\to +\infty}
\scalpro{\bar\signed}{1} \in \RR_{+} \).
As \( \sequence{\belief_n}{n\in\NN} \subset \BELIEF \), where \( \BELIEF \) is a compact set,
we can always suppose (up to relabeling) that there exists \( \bar\belief\in\BELIEF \) such that 
\( \belief_n \to_{n\to +\infty} \bar\belief \in\BELIEF \),
hence that \( \liminf_{n\to +\infty}\fonctiondeux(\belief_n) \geq \fonctiondeux(\bar\belief) \)
since the function \( \fonctiondeux \colon \BELIEF \to \RR \) is \lsc\
(it is even continuous by \cite{Gale-Klee-Rockafellar:1968}).
Then, from \( \bp{\lambda_n\belief_n,\lambda_n\fonctiondeux(\belief_n)+t_n}
\to_{n\to +\infty} \np{\bar\signed,\bar z} \),
  we deduce, on the one hand, that 
  \( \bar\signed =  \lim_{n\to +\infty}\lambda_n\belief_n = \scalpro{\bar\signed}{1}\bar\belief \)
  and, on the other hand, that
  \begin{equation*}  
 \bar z =  \lim_{n\to +\infty}\lambda_n\fonctiondeux(\belief_n)+\underbrace{t_n}_{\geq 0}
  \geq \liminf_{n\to +\infty}\lambda_n\fonctiondeux(\belief_n)
  = \underbrace{\scalpro{\bar\signed}{1}}_{\geq 0} \, \liminf_{n\to +\infty}\fonctiondeux(\belief_n) 
  \geq \scalpro{\bar\signed}{1} \fonctiondeux(\bar\belief)
   \eqfinv 
\end{equation*}
  hence that
\( \bar\signed = \scalpro{\bar\signed}{1}\bar\belief \) and 
\( \bar z \geq \scalpro{\bar\signed}{1} \fonctiondeux(\bar\belief) \).
Then, we consider two cases.
In the case where \( \scalpro{\bar\signed}{1} > 0 \), we get that
\( \np{\bar\signed,\bar z} \in \RR_{++} \epigraph\fonctiondeux \subset \epigraph\widetilde{\fonctiondeux} \)
by~\eqref{eq:epigraph_fonctiondeux}.
In the case where \( \scalpro{\bar\signed}{1} = 0 \), we get that
\( \np{\bar\signed,\bar z}=\np{0,\bar z} \in \na{0}\times\RR_{+} \subset \epigraph\widetilde{\fonctiondeux} \)
by~\eqref{eq:epigraph_fonctiondeux}.
Thus, we have proved that \( \overline{\RR_{++} \epigraph\fonctiondeux} \subset \epigraph\widetilde{\fonctiondeux} \).

Having proven both
\( \epigraph\widetilde{\fonctiondeux} \subset  \overline{\RR_{++} \epigraph\fonctiondeux} \)
and \( \overline{\RR_{++} \epigraph\fonctiondeux} \subset \epigraph\widetilde{\fonctiondeux} \),
we conclude that
\( \epigraph\widetilde{\fonctiondeux} =\overline{\RR_{++}\epigraph\fonctiondeux} \).
We deduce by \cite[Chapter~V, Proposition~1.1.3]{Hiriart-Urruty-Lemarechal-I:1993} 
that the function~\( \widetilde{\fonctiondeux} \) is sublinear and \lsc.

Now, we have finally obtained that the function~\( \widetilde{\fonctiondeux} \)
is proper, sublinear and \lsc.
As a consequence,
by \cite[Chapter~V, Theorem~3.1.1]{Hiriart-Urruty-Lemarechal-I:1993}, 
we get that \( \widetilde{\fonctiondeux} = \SupportFunction{\Primal} \), where
\(  \Primal= \Ba{\primal\in \RR^{\NATURE} \mid \scalpro{\belief}{\primal}
    \leq \widetilde{\fonctiondeux}\np{\signed} \eqsepv \forall \signed\in
    \RR^{\NATURE}} \)
(as the function~\( \widetilde{\fonctiondeux} \) satisfies~\eqref{eq:positively_homogeneous_extension_formula}
it is easy to see that \( \Primal=\Ba{\primal\in \RR^{\NATURE} \mid \scalpro{\belief}{\primal}
  \leq \fonctiondeux\np{\belief} \eqsepv \forall \belief\in\BELIEF } \)).
The subset \( \Primal \) is not empty
because the function~\( \widetilde{\fonctiondeux} \) is proper.
By~\eqref{eq:positively_homogeneous_extension_formula} and~\eqref{eq:support_function}, we conclude that
\begin{equation}
  \fonctiondeux\np{\belief}
=\widetilde{\fonctiondeux}\np{\belief}
=\SupportFunction{\Primal}\np{\belief}
  =  \sup_{\primal\in \Primal}\scalpro{\signed}{\primal} \eqsepv \forall\belief\in\BELIEF
  \eqfinp
\label{eq:value_function_support_function}
    \end{equation}
    To prove Item~\ref{it:additively_separable}, it suffices now to define the
    decision problem~$\MEDIUM$ by $\NEUTRALpayoff_{\MEDIUM} \colon
    \NEUTRAL_{\MEDIUM}\times \NATURE\to \RR$, where
    \begin{equation}
      \NEUTRAL_{\MEDIUM}= \Primal \eqsepv  \NEUTRALpayoff_{\MEDIUM}(\primal,\nature)=
  \primal\np{\nature} \eqsepv \forall \nature \in \NATURE 
  \eqfinv
  \label{eq:above_decision_problem}
    \end{equation}
and to observe that~\eqref{eq:value_function_support_function} is the value
function~\eqref{eq:value_function_utility_classic}
of the above decision problem~\eqref{eq:above_decision_problem}.
\bigskip

This ends the proof.

\end{proof}

\bibliographystyle{alpha}
\bibliography{DeLara,InformationMoreValuable}

\end{document}